\documentclass[a4paper,12pt]{article}
\usepackage[pagewise]{lineno}
\usepackage{amsmath}
\usepackage{amssymb}
\usepackage{mathrsfs}
\usepackage{amsfonts}
\usepackage{amsthm}
\usepackage{pstricks, pst-node, pst-text, pst-3d,psfrag}
\usepackage{graphicx}
\usepackage{indentfirst}
\usepackage{enumerate}
\usepackage{subfigure}
\usepackage{cite}
\usepackage[colorlinks=true]{hyperref}
\hypersetup{urlcolor=blue, citecolor=red}

\theoremstyle{plain}
\newtheorem{thm}{Theorem}[section]
\newtheorem{prop}[thm]{Proposition}

\newtheorem{lem}[thm]{Lemma}
\theoremstyle{definition}

\theoremstyle{remark}
\newtheorem{rmk}{\textnormal{{\bf Remark}}}[section]

\numberwithin{equation}{section}

\makeatletter 
\@addtoreset{equation}{section}
\makeatother  

\newcommand\RR{\ensuremath{\mathbb{R}}}

\newcommand\ep{\epsilon}
\newcommand\om{\omega}

\begin{document}

\title{ Nonexistence of observable chaos and its robustness in strongly monotone dynamical systems\thanks{Supported by NSF of China No.11825106 and 12090012.}}

\setlength{\baselineskip}{16pt}

\author {
Yi Wang and Jinxiang Yao\thanks{Corresponding author: jxyao@mail.ustc.edu.cn (J. Yao).}
\\[2mm]
School of Mathematical Sciences\\
University of Science and Technology of China\\
Hefei, Anhui, 230026, P. R. China
}
\date{}
\maketitle

\begin{abstract}
For strongly monotone dynamical systems on a Banach space, we show that the largest Lyapunov exponent  $\lambda_{\max}>0$ holds on a shy set in the measure-theoretic sense. This exhibits that strongly monotone dynamical systems admit no observable chaos, the notion of which was formulated by L.S. Young. We further show that such phenomenon of no observable chaos is robust under the $C^1$-perturbation of the systems.

\end{abstract}

\section{Introduction}
\emph{Observable chaos} means that the largest Lyapunov exponent $\lambda_{\max}>0$ holds on a positive Lebesgue measure set (see, Young
\cite[Section 1]{Y13},\cite[Section 4]{Y13B} or \cite[Section 3]{Y18}). The occurrence of observable chaos implies that the instability persists for all future times and happens  on a set large enough to be observable. It is actually a stronger form of chaos than the presence
of the complicated chaotic behavior (say, horseshoe) alone. As a matter of fact, for the horseshoe itself occupies a zero measure set, its presence does not preclude the possibility that
orbits starting from a full measure set may tend eventually to a stable equilibrium.
Such unobservable complicated dynamics are sometimes characterized as {\it transient chaos} (see e.g., Young \cite[Section 1, p.1442]{Y13}). It implies that orbits starting near a horseshoe may appear chaotic for a short time (as they follow orbits within the horseshoe) before tending to a tamed behavior. Transient chaos is ubiquitous in fluid, chemical, biological, and engineering systems. One may refer to \cite{ASY97,Ji21,LQ10,LT11,T15} and references therein for more details of the theory and applications of transient chaos and transient dynamics.

The theory of monotone dynamical systems grew out of the series of ground-breaking work of M. W. Hirsch (\cite{H85,HS05,H84,H88}). Since the profound work of Smale \cite{S76}, it has been known that monotone dynamical systems can certainly possess orbits that behave in chaotic way, including horseshoes (see e.g., \cite{DP98,S76,S88,S95,S97,WJ01,WX10}). On the other hand, the celebrated generic/prevalent convergence (see e.g., \cite{EHS08,H85,HS05,P89,PT92,ST91,T94,WY20-1,WY21-2,WYZ20-2}) make us believe that the chaotic behavior is so unstable as to be unobservable in monotone dynamical systems. However, there is a lack of rigorous research in the existing literatures on the observability of chaos in monotone dynamical systems.

In the present paper, we shall focus on and prove the nonexistence of observable chaos (in the sense of Young \cite[Section 1]{Y13}) for $C^1$-smooth strongly monotone dynamical systems on a Banach space. More precsisely, we will show that the largest Lyapunov exponent  $\lambda_{\max}>0$ (see the Young's definition in \eqref{E:map-lar-lya} in Section 2) holds on a shy set in the measure-theoretic sense. Furthermore, we will also show that such phenomenon of no observable chaos is robust under the $C^1$-perturbation of the systems.

To this purpose, we formulate some standing hypotheses:

\vskip 2mm

\noindent \textbf{(H1)} $(X,C)$ is a strongly ordered separable Banach Space.
\vskip 1mm

\noindent \textbf{(H2)} $F_{0}:X \to X$ is a compact ${C}^{1}$-map, such that for any $x\in X$, the Fr\'echet derivative $DF_{0}(x)$

\quad\  is a strongly positive operator, i.e., $DF_{0}(x)v\gg0$ whenever $v>0$.

\newtheorem*{thma}{\textnormal{\textbf{Theorem A}}}
\begin{thma}[No observable chaos for monotone mappings]
\emph{Assume that \textnormal{(H1)-(H2)} hold. Assume also $F_{0}$ is pointwise dissipative with an attractor $A$. Then the set
$$S_{0}:=\{x\in X:\ \lambda_{\max}(x,F_{0}) >0\},$$
is  shy in $X$. In particular, if $X=\mathbb{R}^{d}$, then the set $S_{0}$ is of Lebesgue measure  zero  in $\mathbb{R}^{d}$; and hence, $F_{0}$ cannot have observable chaos.}
\end{thma}

Here, we call a subset $W\subset X$ \emph{shy} if there exists a nonzero compactly supported Borel measure $\mu$ on $X$ such that $\mu(x+W)=0$ for every $x\in X$. Roughly speaking, ``shyness" describes properties of interest occurs for ``almost nowhere" in an infinite-dimensional space from a probabilistic or measure-theoretic perspective. It is a natural generalization to separable Banach spaces of the notion of the Lebesgue measure  zero  for Euclidean spaces (see Section \ref{S:Not-pre}). In particular, on $\mathbb{R}^d$, it is equivalent to the notion of ``Lebesgue almost nowhere".  More properties of shy sets will be mentioned in Section 2.

\vskip 2mm
Motivated by Theorem A, we further consider the $C^1$-perturbed system $\{F_\epsilon^n\}_{n\in \mathbb{N}}$ ($F_\epsilon$ not necessarily monotone), and obtain the nonexistence of observable chaos for the $C^1$-perturbed systems.
More precisely, we present an additional standing hypothesis:

\vskip 3mm
\noindent \textbf{(H3)} Let $J=[-\epsilon_{0},\epsilon_{0}]\subset\mathbb{R}$, and $F:J\times X\to X; (\epsilon,x)\mapsto F_{\epsilon}(x)$ is a compact $C^{1}$-map, i.e., $DF(\epsilon,x)$ continuously depends on $(\epsilon,x)\in J\times X$.
\vskip 3mm

The following theorem reveals that the nonexistence of observable chaos of $F_0$ is robust under the $C^1$-perturbation.

\newtheorem*{thmb}{\textnormal{\textbf{Theorem B}}}
\begin{thmb}[$C^1$-robustness for nonexistence of observable chaos]
\emph{Assume that \textnormal{(H1)-(H3)} hold. Assume also $F_{0}$ is pointwise dissipative with an attractor $A$. Let $B_{1}\supset A$ be an open ball such that
\begin{equation}\label{E:C1-close}
\sup\{\Vert F_{\epsilon}x-F_{0}x\Vert +\Vert DF_{\epsilon}(x)-DF_{0}(x)\Vert : \epsilon\in J,\,x\in B_{1}\}
\end{equation}
sufficiently small.  Then there exists a closed bounded set $M$ (whose interior contains $A$), an integer $q>0$ and an $\epsilon_1>0$ such that, for each $\epsilon\in(-\epsilon_{1},\epsilon_{1})$,}

\textnormal{(i)} \emph{$F_{\ep}^{q}(M)\subset M$; and}

\textnormal{(ii)} \emph{The set}
$$S_{\ep}:=\{x\in M:\ \lambda_{\max}(x,F_{\epsilon})> 0\},$$
\emph{is shy in $M$. In particular, if $X=\mathbb{R}^{d}$, then $S_{\ep}$ is of Lebesgue measure  zero  in $M$.}
\end{thmb}
\vskip 2mm

Now, we focus on a strongly monotone semiflow $\Phi$ on $X$, we present the following standing hypothesis:
\vskip 2mm

\noindent $\textbf{(H2)}^{\prime}$ $\Phi:\RR\times X\to X$ is a $C^1$-smooth strongly monotone semiflow with compact orbit closures. For some fixed  $t_0>0$, the Fr\'echet derivative $D\Phi_{t_0}(x)$ is compact for any $x\in X$ and $D\Phi_{t_0}(e)$ is strongly positive for any equilibrium $e\in X$.

\newtheorem*{thmc}{\textnormal{\textbf{Theorem C}}}
\begin{thmc}[No observable chaos for semiflows]
\emph{Assume that \textnormal{(H1)} and $\textnormal{(H2)}^{\prime}$ hold. Then the set
$$S_{\Phi}:=\{x\in X:\ \lambda_{\max}(x,\Phi)> 0\},$$
is shy in $X$. In particular, if $X=\mathbb{R}^{d}$, then the set $S_{\Phi}$ is of Lebesgue measure  zero  in $\mathbb{R}^{d}$; and hence, $\Phi$ cannot have observable chaos.}
\end{thmc}

Theorems A-C conclude that there is no observable chaos in $C^1$-smooth strongly monotone systems, as well as for $C^1$-perturbed systems.
It is known that monotone systems cannot have chaotic attractors in the sense of topological transitivity (see e.g., \cite{H84,H85attractor,H19}). Here, Theorems A-C provide us a new point of view in measure-theoretic sense to understand such unstable complicated dynamics in monotone dynamical systems.

It deserves to point out that the aforementioned unobservable complicated dynamics in strongly monotone systems are indeed transient chaos (in the sense of Young \cite[Section 1, p.1442]{Y13}).
The nonexistence of observable chaos and the possible existence of transient chaos exhibit an interaction relationship between the generic/prevalent asymptotic behavior (in the whole space) and arbitrary complicated dynamics (in the codimension-1 invariant manifolds) in monotone systems (see also Remark \ref{R:im-pre-tran-chaos}).

Our approach for proving  Theorems A-C is motivated by the \emph{Sharpened $C^1$-dynamics alternative} (Lemma \ref{P:im-da}) and its $C^1$-robustness (Lemma \ref{P:im-da-perturb}) in our recent work \cite{WY21-2}. By appealing to the Sharpened $C^1$-dynamics alternative and its $C^1$-robustness, and a useful tool of upper (lower) $\omega$-unstable sets introduced by Tak\'{a}\v{c} \cite{Ta92}, we obtain the \emph{Improved Prevalent Dynamics} theorem (Proposition \ref{T:im-pre}) and its $C^1$-robustness (Proposition \ref{T:im-pre-per}). Together with a formula for the spectral radius of an operator (Lemma \ref{L:spe-radi-norm}), we utilize the improved prevalent dynamics theorem to accomplish our approach.

This paper is organized as follows. In Section \ref{S:Not-pre}, we agree on some notations, give relevant definitions and preliminary results. In Section \ref{S:im-pre}, we establish the improved prevalent dynamics theorem (Proposition \ref{T:im-pre}) and its $C^1$-robustness (Proposition \ref{T:im-pre-per}), which turn out to be crucial in the proof our main results (Theorems A-C). Section \ref{S:tran-chaos} is devoted to the proof of the main results.

\section{Notations and Preliminary results}\label{S:Not-pre}

Let $(X,\|\cdot\|)$ be a Banach space and $\mathscr{L}(X)$ the space of all bounded linear operators which map $X$ into itself. A cone $C$ is a closed convex subset of $X$ such that $\lambda C\subset C$ for all $\lambda>0$ and $C\cap(-C)=\{0\}$. $(X,C)$ is said to be a strongly ordered Banach space if $C$ has nonempty interior ${\rm Int}C$. For $x,y\in X$, we write $x\leq y$ if $y-x\in C$, $x<y$ if $y-x\in C\setminus\{0\}$, $x\ll y$ if $y-x\in{\rm Int}C$. The reversed signs are used in the usual way.  A subset $J^{\prime}\subset X$ is called a \emph{simply ordered, open arc} if there is an increasing homeomorphism $h$ from an open interval $I\subset\mathbb{R}$ onto $J^{\prime}$ ($h$ is \emph{increasing} if $\xi_{1}<\xi_{2}$ implies $h\xi_{1}\ll h\xi_{2}$). Let $D\subset X$ be an open subset.  A mapping $h:D\to D$ is \emph{monotone} (\emph{strongly monotone}), if $x\leq y$ ($x<y$) implies $hx\leq hy$ ($hx\ll hy$). Similarly, a semiflow $\Phi:\RR\times X\to X$ is \emph{monotone} (\emph{strongly monotone}), if $x\leq y$ ($x<y$) implies $\Phi_{t}(x)\leq \Phi_{t}(y)$ ($\Phi_{t}(x)\ll \Phi_{t}(y)$) for all $t\geq0$ ($t>0$).

In this paper, we sometimes also need to deal with arguments for another solid cone $C_1(\subset C)$. Hence, for the sake of no confusion, we write $\leq_1, <_1, \ll_1$ as the corresponding order relation induced by the cone $C_1$ throughout the paper.

Let $D\subset X$, we denote by $\overline{D}$ the closure of $D$. For a continuous map $h:X\to X$, the orbit of $x\in X$ is denoted by $O(x,h)$=$\{h^{n}x:n\geq0\} $. The $\omega$-limit set of $x\in X$ is $\omega(x,h)=\mathop{\bigcap}\limits_{k\geq0}\overline{\{h^{n}x:n\geq k\}}$. We say that $h$ is $\om$-compact in a subset $Y$ of $X$, if $O(x,h)$ is relatively compact for each $x\in Y$ and $\mathop{\bigcup}\limits_{x\in Y}\om(x,h)$ is relatively compact. Given any $x\in X$, we define the \emph{upper} and  \emph{lower $\omega$-limit sets} of $x$ by
$$\omega_{+}(x,h):=\bigcap_{\substack{u\in X\\u\gg x}}\overline{\bigcup_{\substack{y\in X\\ u\geq y>x}}\omega(y,h)} \quad and\quad
\omega_{-}(x,h):=\bigcap_{\substack{u\in X\\u\ll x}}\overline{\bigcup_{\substack{y\in X\\ u\leq y<x}}\omega(y,h)},$$
respectively. If $h$ is $\omega$-compact in some neighbourhood of $x$ in $X$, $\omega_{+}(x,h)$ (resp. $\omega_{-}(x,h)$) is non-empty and compact  (see Tak\'{a}\v{c} \cite[Proposition 3.1]{Ta92}). Write
$$\mathcal{U}_{+}(h)=\{x\in X:\omega_{+}(x,h)\neq \omega(x,h)\} \quad and\quad \mathcal{U}_{-}(h)=\{x\in X:\omega_{-}(x,h)\neq \omega(x,h)\}$$
as the \emph{upper $\omega$-unstable set} and \emph{lower $\omega$-unstable set} in $X$, respectively. Moreover, we denote the set $\mathcal{U}_2(h)=\mathcal{U}_{+}(h)\cap\mathcal{U}_{-}(h)$.

Now, we give a structure proposition of $\mathcal{U}_{-}$ ($\mathcal{U}_{+}$, resp.), which is useful in Section \ref{S:im-pre}.

\begin{prop}\label{P:unsta-count}
Let $D\subset X$ be an open subset and $h:D\to D$ be strongly monotone. Assume that the set
$$D_0=\{x\in D:\ O(x,h)\text{ is relatively compact }\}$$
is dense in $D$. Let $J^{\prime}\subset D_0$ be a simply ordered, open arc. Then the set $J^{\prime}_{-}=J^{\prime}\cap\mathcal{U}_{-}(h)$ is at most countable. A corresponding result holds for $\mathcal{U}_{+}(h)$.
\end{prop}

\begin{proof}
See \cite[Corollary 3.4]{Ta92}.
\end{proof}

A point $x\in X$ is called a \emph{periodic point of $h$}, if $h^{p}x=x$ for some integer $p\geq1$. $p$ is then a \emph{period} of $x$. Moreover, if $h^{l}x\neq x$ for $l=1,2,\cdots,p-1$, $x$ is said to be \emph{$p$-periodic}. $p$ is the \emph{minimal period} of $x$. Particularly, if $p=1$, $x$ is a \emph{fixed point} of $h$. $K$ is a \emph{cycle} if $K=O(x,h)$ for some periodic point $x$. For a $C^1$-smooth map $h$ and $x\in X$, we define
\begin{equation}\label{E:map-prin-lya}
\lambda_{1}(x,h)={\limsup_{n \to +\infty}\frac{\log\|Dh^{n}(x)\|}{n}}
\end{equation}
as the  \emph{principal Lyapunov exponent of x} (with respect to $h$). We say a cycle $K=O(x,h)$ (of minimal period $p$) is \emph{linearly stable} if $r_{\sigma}(Dh^p(x))\leq1$, where $r_{\sigma}(Dh^p(x))$ is the spectral radius of $Dh^p(x)$ ($x$ is also said to be a linearly stable $p$-periodic point of $h$). This includes the neutral case (the principal eigenvalue on the unit circle). It is equivalent to $\lambda_{1}(x,h)\leq0$ (see e.g., Hess and Pol\'{a}\v{c}ik \cite[p.1316-1317]{PH93}). Particularly, if $p=1$, $x$ is called a linearly stable fixed point of $h$. Let $B\subset X$. $k$ is said to be a \emph{stable period} for the restriction ${h|}_{B}$ if there is a linearly stable $k$-periodic point $x$ of $h$ such that the orbit $O(x,h)\subset B$. If $B=X$, we simply say that $k$ is a stable period of $h$. For brevity, we hereafter say $\omega(x,h)$ is a linearly stable cycle (of minimal period $p$), if $\omega(x,h)$ is a linearly stable cycle (of minimal period $p$) of $h$.

A continuous map $h:X\to X$ is \emph{pointwise dissipative} if there is a bounded subset $B\subset X$ such that $B$ attracts each point of $X$. An invariant set $A$ is called an \emph{attractor} of $h$ if $A$ is the maximal compact invariant set which attracts each bounded subset of $X$. If $h:X\to X$ is compact and pointwise dissipative, then there is a connected attractor $A$ of $h$ (see e.g., \cite[Theorem 2.4.7]{Ha88}).

For a semiflow $\Phi:\RR\times X\to X$, the orbit of $x\in X$ is $O(x,\Phi)$=$\{\Phi_{t}x:t\geq0\}$. The $\omega$-limit set of $x\in X$ is $\omega(x,\Phi)=\mathop{\bigcap}\limits_{s\geq0}\overline{\{\Phi_{t}x:t\geq s\}}$.  A point $e\in X$ is called an \emph{equilibrium} if $\Phi_{t}(e)=e$ for any $t\geq0$. We call a semiflow $\Phi$ is $C^1$-smooth if the Fr\'echet derivative $D\Phi_{t}(x)$ with respect to the state variable $x$ exists for each $x\in X$ and $t>0$, and $x\mapsto D\Phi_{t}(x)$ is continuous. For a $C^1$-smooth semiflow $\Phi$ and $x\in X$, we define the  \emph{principal Lyapunov exponent of x} (with respect to $\Phi$) as
\begin{equation}\label{E:flow-prin-lya}
\lambda_{1}(x,\Phi)={\limsup_{n \to +\infty}\frac{\log\|D\Phi_{t}(x)\|}{t}}.
\end{equation}
An equilibrium $e$ is called \emph{linearly stable}, if $\lambda_{1}(e,\Phi)\leq0$.

Now, we give a helpful lemma on the minimal period of linearly stable periodic points.
\begin{lem}\label{L:cycle-fixed pt}
Let $h:X\to X$ be a $C^1$-map. If $z$ is a linearly stable $k$-periodic point of $h^q$, then $z$ is a linearly stable periodic point of $h$ of minimal period at most $kq$.
\end{lem}

\begin{proof}
See the claim in the proof of \cite[Corollary 2.5]{WY21-2}.
\end{proof}

In the following, we introduce the concept of {\it largest Lyapunov exponent} by Young \cite[Section 1, p.1441]{Y13}. For a $C^1$-smooth map $h:X\to X$ and $x\in X$, we define
\begin{equation}\label{E:map-lar-lya}
\lambda_{\max}(x,h)={\liminf_{n \to +\infty}\frac{\log\|Dh^{n}(x)\|}{n}}
\end{equation}
as the \emph{largest Lyapunov exponent of x} (with respect to $h$).
Similarly, the \emph{largest Lyapunov exponent of x} (with respect to a semiflow $\Phi$) is defined as
\begin{equation}\label{E:flow-lar-lya}
\lambda_{\max}(x,\Phi)={\liminf_{t \to +\infty}\frac{\log\|D\Phi_{t}(x)\|}{t}}.
\end{equation}
It is easy to see from \eqref{E:map-prin-lya}-\eqref{E:flow-lar-lya} that
\begin{equation}\label{E:map-lambda 1-geq-lambda max}
\lambda_{1}(x,h)\geq\lambda_{\max}(x,h)
\end{equation}
and
\begin{equation}\label{E:flow-lambda 1-geq-lambda max}
\lambda_{1}(x,\Phi)\geq\lambda_{\max}(x,\Phi).
\end{equation}

Let $M$ be a compact domain (the closure of a non-empty connected open set) of $\mathbb{R}^{d}$ or a finite-dimensional Riemannian manifold, and $h: M\to M$ be a ${C}^{1}$-map. $h$ is said to have \emph{observable chaos} if $\lambda_{\max}(\cdot,h)>0$ holds on a positive Lebesgue measure set (see, Young
\cite[Section 1, p.1441]{Y13},\cite[Section 4]{Y13B},\cite[Section 3]{Y18}).

Similarly, We call  a $C^1$-smooth semiflow $\Phi:\RR\times M\to M$ have \emph{observable chaos}, if $\lambda_{\max}(\cdot,\Phi)>0$ holds on a positive Lebesgue measure set.

In what follows, we introduce the definition and some significant properties of \emph{prevalence} and \emph{shyness}. A Borel subset $W\subset X$ is called \emph{shy} if there exists a nonzero compactly supported Borel measure $\mu$ on $X$ such that $\mu(W+x)=0$ for every $x\in X$. A Borel subset $W\subset X$ is \emph{prevalent (in $X$)} if its complement $X\backslash W$ is shy. Given a Borel subset $V\subset X$, we say that a Borel subset $W\subset X$ is \emph{prevalent in} $V$ if $V\setminus W$ is shy.

Shyness has the following fundamental properties (\cite{HSY92,HK10}):

(i) Every Borel subset of a shy set is shy;

(ii) Every translation of a shy set is shy;

(iii) No nonempty open set is shy;

(iv) Every countable union of shy sets is shy;

(v) In finite-dimensional spaces, a Borel set $W$ is shy if and only if it has Lebesgue measure

\quad \ \ \,zero.

At the end of this section, we present a sufficient condition which guarantees a Borel subset $W\subset X$ to be shy.
\begin{prop}\label{P:Borel-suffi-condi}
Let $W\subset X$ be a Borel subset and assume that there exists $v\gg0$ such that $L\cap W$ is countable for every straight line $L$ parallel to $v$. Then $W$ must be shy in $X$.
\end{prop}

\begin{proof}
See Enciso, Hirsch and Smith \cite[Lemma 1]{EHS08}.
\end{proof}


\section{Improved prevalent dynamics/and its $C^1$-robustness for $C^1$-perturbed Systems}\label{S:im-pre}

In this section, we will focus on the proof of the \emph{Improved prevalent dynamics} theorem (Proposition \ref{T:im-pre}), as well as its $C^1$-robustness for $C^1$-perturbed systems (Proposition \ref{T:im-pre-per}).
Proposition \ref{T:im-pre} and Proposition \ref{T:im-pre-per} will play a crucial role in our approach for the nonexistence of observable chaos and its $C^1$-robustness in Section \ref{S:tran-chaos}.

\begin{prop}\label{T:im-pre}
{\rm (Improved Prevalent Dynamics).} Assume that \textnormal{(H1)-(H2)} hold. Assume also $F_{0}$ is pointwise dissipative with an attractor $A$. Then there is an integer $m>0$ such that the set
$$Q_{0}:=\{x\in X:\omega(x,F_{0})\text{ is a linearly stable cycle of minimal period at most }m\}$$
is prevalent in $X$. In particular, if $X=\mathbb{R}^{d}$, then the set $Q_0$ is of full Lebesgue measure in $X$.
\end{prop}

\begin{rmk}\label{R:Wang et al.}
Enciso, Hirsch and Smith \cite{EHS08} firstly investigated the prevalent behavior of strongly monotone semiflows and proved that the set of points that converge to equilibria is prevalent. Recently, Wang et al. \cite{WYZ20-2} proved the set of points that converge to cycles is prevalent for discrete-time systems.
Here, Proposition \ref{T:im-pre} improves \cite[Theorem A]{WYZ20-2} by showing that the set of minimal periods of linearly stable cycles is bounded by $m$. In addition, we succeed in proving the measurability of $Q_0$ in Proposition \ref{T:im-pre} by virtue of the boundedness of stable periods. Moreover, Proposition \ref{T:im-pre} is also shown to be robust under the $C^{1}$-perturbation (see Proposition \ref{T:im-pre-per}).
\end{rmk}

We further consider the $C^1$-robustness for the improved prevalent dynamics. The following theorem reveals that the improved prevalent dynamics of $F_0$ (Proposition \ref{T:im-pre}) is robust under the $C^1$-perturbation.

\begin{prop}\label{T:im-pre-per}
{\rm ($C^1$-robustness for improved prevalent dynamics).} Assume that \textnormal{(H1)-(H3)} hold. Assume also $F_{0}$ is pointwise dissipative with an attractor $A$. Let $B_{1}\supset A$ be an open ball such that
$$\sup\{\Vert F_{\epsilon}x-F_0x\Vert +\Vert DF_{\epsilon}(x)-DF_0(x)\Vert :\epsilon\in J, x\in B_1\}$$ sufficiently small.
Then there exists an open bounded set $D_1\supset A$, an integer $m>0$ and an $\epsilon_1>0$ such that, for each $\epsilon\in(-\epsilon_{1},\epsilon_{1})$, the set
$$Q_{\epsilon}:=\{x\in D_1:\omega(x,F_{\epsilon})\text{ is a linearly stable cycle of minimal period at most }m\}$$
is prevalent in $D_1$. In particular, if $X=\mathbb{R}^{d}$, then the set $Q_{\epsilon}$ is of full Lebesgue measure in $D_1$.
\end{prop}

In order to prove Proposition \ref{T:im-pre}, we need the following important sharpened dynamics alternative for $C^1$-smooth strongly monotone mapping.

\begin{lem}\label{P:im-da}
{\rm (Sharpened $C^1$-dynamics alternative).} Assume that \textnormal{(H1)-(H2)} hold. Assume also $F_{0}$ is pointwise dissipative. Then there is an integer $m>0$ such that, for any $x\in X$, either

\textnormal{(a)} $\omega(x,F_{0})$ is a linearly stable cycle of minimal period at most $m$; or,

\textnormal{(b)} there is a constant $\delta>$ 0 such that, for any $y \in X$ satisfying $y < x$ or $y > x$,
$$\mathop{\limsup}\limits_{n \to +\infty}\|F_{0}^{n}x-F_{0}^{n}y\|\geq\delta.$$
\end{lem}

\begin{proof}
See Wang and Yao \cite[Theorem A]{WY21-2}.
\end{proof}

\begin{rmk}\label{R:bdd-per-F0}
In fact, all the stable periods of $F_{0}$ are bounded above by $m$ (see \cite[Corollary 2.5]{WY21-2}).
\end{rmk}


\begin{lem}\label{L:borel-lsequi}
Let $B\subset X$ be a Borel subset and $h:B\to B$ be a $C^1$-map. Then the set $C_{s}:=\{x\in B: \omega(x,h)=\{z\}\ \text{for some linearly stable fixed point $z$ of $h$ }\}$ is Borel.
\end{lem}

\begin{proof}
Enciso, Hirsch and Smith \cite[Appendix, Lemma 10]{EHS08} proved this lemma for $C^1$-semiflows. One can follow the exactly same arguments and obtain an analogous proof for $C^1$-maps. We omit it here.
\end{proof}

Now, we are ready to prove Proposition \ref{T:im-pre}.
\vskip 2mm
\noindent
{\it Proof of Proposition \ref{T:im-pre}.}
Let the integer $m>0$ be in Lemma \ref{P:im-da}. Define
$$Q_{0}:=\{x\in X:\omega(x,F_{0})\text{ is a linearly stable cycle of minimal period at most }m\},$$
and
$$Q_{0}^{\prime}:=\{x\in X:\omega(x,F_{0}^{m!})\text{ is a linearly stable fixed point}\}.$$
We first show that $Q_{0}=Q_{0}^{\prime}$.
Obviously, $Q_{0}\subset Q_{0}^{\prime}$. Notice from Remark \ref{R:bdd-per-F0} that all the stable periods of $F_{0}$ are bounded above by $m$. Then, $Q_{0}^{\prime}\subset Q_{0}$. Thus, $Q_{0}=Q_{0}^{\prime}$.

It follows from Lemma \ref{L:borel-lsequi} that $Q_{0}^{\prime}$ is Borel. So, $Q_{0}$ is also Borel. By virtue of Lemma \ref{P:im-da}, we can repeat the exactly same arguments in Pol\'{a}\v{c}ik and Tere\v{s}\v{c}\'{a}k \cite[Section 5]{PT92} to obtain that $X\setminus Q_{0}\subset\mathcal{U}_{2}(F_{0})$. Proposition \ref{P:unsta-count} implies that, for every simply ordered, open arc $J^{\prime}$, $J^{\prime}\cap\mathcal{U}_{2}(F_{0})$ is at most countable. So, $J^{\prime}\setminus Q_{0}$ is also at most countable. Thus, Proposition \ref{P:Borel-suffi-condi} entails that $X\setminus Q_{0}$ is shy. Therefore, $Q_{0}$ is prevalent.

In particular, if $X=\mathbb{R}^{d}$, then $Q_{0}$ is of full Lebesgue measure in $X$. We have completed the proof.\hfill $\square$
\vskip 3mm

In the rest of this section, we are going to prove Proposition \ref{T:im-pre-per}. To this end, we introduce the following lemma, which reveals that the sharpened dynamics alternative of $F_0$ (Lemma \ref{P:im-da}) is robust under the $C^1$-perturbations.

\begin{lem}\label{P:im-da-perturb}
{\rm ($C^1$-robustness for sharpened $C^1$-dynamics alternative).}
Let all hypotheses in Proposition \ref{T:im-pre-per} hold.
Then there is a solid cone $C_1\subset{\rm Int}C$, an open bounded set $D_1$ ($B_{1}\supset D_1\supset A$), integers $q, m_1>0$ and an $\epsilon_1>0$ such that, for each $\epsilon\in(-\epsilon_1,\epsilon_1)$,

{\rm (i)}. For each $n\geq q$, $F_{\epsilon}^{n}(D_1)\subset D_1$ and $F_{\epsilon}^{n}x\ll_1F_{\epsilon}^{n}y$ whenever $x<_1y$ (with $x, y\in D_1$).

{\rm (ii)}. For each $x\in D_1$, either

\quad\quad\textnormal{(a)} $\omega(x,F_{\epsilon}^{q})$ is a linearly stable cycle of minimal period at most $m_1$; or,

\quad\quad\textnormal{(b)} there is a constant $\delta>$ 0 such that, for any $y \in D_1$ satisfying $y <_1x$ or $y >_1 x$,$$\mathop{\limsup}\limits_{n \to +\infty}\|F_{\epsilon}^{nq}x-F_{\epsilon}^{nq}y\|\geq\delta.$$
\end{lem}

\begin{proof}
For the proof of item (i), we refer to Tere\v{s}\v{c}\'{a}k \cite[Theorem 5.1]{T94}. For the proof of item (ii), see \cite[Theorem 3.1]{WY21-2}.
\end{proof}

\begin{rmk}\label{R:bdd-per-Fep}
In fact, for each $\epsilon\in(-\epsilon_1,\epsilon_1)$, all the stable periods of ${F_{\epsilon}^{q}|}_{\overline{D}_1}$ are bounded above by $m_1$ (see \cite[Proposition 2.4]{WY21-2}). In addition, for any open bounded subset $D_2$ (satisfying $D_1\supset\overline{D}_2\supset D_{2}\supset A$), one can let the integer $q>0$ larger (if necessary) in Lemma \ref{P:im-da-perturb} such that, both $D_1$ and $D_2$ satisfies items (i)-(ii) in Lemma \ref{P:im-da-perturb} (see \cite[Remark 2.3]{WY21-2} or \cite[Eq.(5.11) on p.19]{T94}). Moreover, one can also choose $D_2$ to be connected, since $A$ is connected. If $X=\mathbb{R}^d$, this leads $\overline{D}_2$ to be a compact domain.
\end{rmk}

Hereafter in this paper, we always reserve the open bounded subsets $D_1,D_2$ (with $D_1\supset\overline{D}_2\supset D_{2}\supset A$) and the integers $q,m_1>0$ as in Lemma \ref{P:im-da-perturb} and Remark \ref{R:bdd-per-Fep}.

Now, we are ready to prove Proposition \ref{T:im-pre-per}.
\vskip 2mm
\noindent
{\it Proof of Proposition \ref{T:im-pre-per}.}
Let the open bounded subset $D_1$, the integers $q,m_1>0$ and $\epsilon_1>0$  be in Lemma \ref{P:im-da-perturb}. Clearly, it follows from (H3) that $F_\ep^q:D_1\to D_1$ is compact. Moreover, for each $\epsilon\in(-\epsilon_{1},\epsilon_{1})$, Lemma \ref{P:im-da-perturb}(i) directly implies that $F_\ep^q:D_1\to D_1$ is {\it strongly monotone with respect to $C_{1}$}.  By Lemma \ref{P:im-da-perturb}(ii), we can repeat the exactly same arguments in the proof of Proposition \ref{T:im-pre} (with $F_0$ replaced by $F_\ep^q$ there) to obtain that, the set
\begin{equation}\label{E:Fep-m-pre}
\tilde{Q}_{\epsilon}:=\{x\in D_{1}:\omega(x,F_{\epsilon}^{q})\text{ is a linearly stable cycle of minimal period at most }m_1\}
\end{equation}
is prevalent in $D_1$, for each $\epsilon\in(-\epsilon_1,\epsilon_1)$.

Now, we define
$$Q_{\epsilon}:=\{x\in D_{1}:\omega(x,F_{\epsilon})\text{ is a linearly stable cycle of minimal period at most }m\},$$
where $m=m_1q$. On one hand, it is clear that if $\omega(x,F_{\epsilon})$ is a linearly stable cycle, then $\omega(x,F_{\epsilon}^{q})$ is a linearly stable cycle.
Then, the bound $m_1$ of stable periods of ${F_{\epsilon}^{q}|}_{\overline{D}_1}$ in Remark \ref{R:bdd-per-Fep} entails that, $Q_{\epsilon}\subset\tilde{Q}_{\epsilon}$. On the other hand, it follows from Lemma \ref{L:cycle-fixed pt} that $\tilde{Q}_{\epsilon}\subset Q_{\epsilon}$. Then, we have proved $Q_{\epsilon}=\tilde{Q}_{\epsilon}$. Thus, $\tilde{Q}_{\epsilon}$ is also prevalent in $D_{1}$, for each $\epsilon\in(-\epsilon_1,\epsilon_1)$.

In particular, if $X=\mathbb{R}^{d}$, then $Q_{\ep}$ is of full Lebesgue measure. This completes the proof.
\hfill $\square$
\vskip 3mm

\begin{rmk}
We specially point out that, one can apply our theoretical result in this section (Proposition \ref{T:im-pre} and Proposition \ref{T:im-pre-per}) to obtain that, the improved prevalence of convergence to periodic solutions with a uniform bound of minimal periods, for time-periodic parabolic equations and their $C^1$-perturbed systems (see e.g. the equations in \cite{PH93,PT93,WY21-2,WYZ20-2}).
\end{rmk}

\section{Proof of the Main Results}\label{S:tran-chaos}
We focus on in this section and prove Theorems A-C. In order to prove our main results, we need the following formula for the spectral radius of an operator.

\begin{lem}\label{L:spe-radi-norm}
Let $(X,\|\cdot\|)$ be a Banach space and $T\in\mathscr{L}(X)$. Then $r_{\sigma}(T)=\inf |T|$, where the infimum is taken over all norms $|\cdot|$ on $X$ equivalent to $\|\cdot\|$.
\end{lem}

\begin{proof}
See Holmes \cite[Theorem on p.164]{Holmes}. For finite-dimensional case, see also Horn and Johnson \cite[Lemma 5.6.10]{Johnson}.
\end{proof}

Now, we are in position to prove Theorems A-C.
\vskip 2mm

\noindent
{\it Proof of Theorem A.}
Let the integer $m>0$ be obtained in Proposition \ref{T:im-pre}. Take $T:=F_{0}^{m!}$. We first {\it assert} that: \emph{If $x\in X$ with $\omega(x,T)=\{z\}$ for some linearly stable fixed point $z$ of $T$, then $\lambda_{1}(x,F_{0})\leq0$}.

In fact, one has $r_{\sigma}(DT(z))\leq1$, since $z$ is a linearly stable fixed point of $T$. Then by Lemma \ref{L:spe-radi-norm}, for any $\tilde{\ep}>0$, there exists a norm $|\cdot|$  equivalent to $\|\cdot\|$ on $X$ such that
$$|DT(z)|<r_{\sigma}(DT(z))+\tilde{\ep}.$$
Hence, $|DT(z)|<1+\tilde{\ep}$. Recall that $T$ is $C^1$. Then there exists a $\delta>0$ such that,
$$|DT(y)|<1+2\tilde{\ep}, \text{ for any } y\in X \text{ with } |y-z|<\delta.$$
It follows from $\omega(x,T)=\{z\}$ that, there exists an integer $N>0$ such that, $|T^nx-z|<\delta$ for any $n\geq N$. Thus, the chain rule shows that
\begin{alignat*}{2}
\lambda_{1}(x,T)&=\limsup_{n \to +\infty}\frac{\log|DT^{n}(x)|}{n}\leq\limsup_{n \to +\infty}\frac{\log|DT^{N}(x)|+\log|DT^{n-N}(T^Nx)|}{n}\\
&\leq\limsup_{n \to +\infty}\frac{\log|DT^{N}(x)|+(n-N)\log(1+2\tilde{\ep})}{n}\leq\log(1+2\tilde{\ep}).
\end{alignat*}
Since the arbitrary of $\tilde{\ep}$, we have $\lambda_{1}(x,T)\leq0$.

For any integer $n\geq1$, write $n=k_{n}m!+l_{n}$, where $k_{n}\geq0$ and $l_{n}\in\{0,1,2,\cdots,m!-1\}$. Then
\begin{eqnarray*}\label{E3.2}
\lambda_{1}(x,F_{0})&=&{\limsup_{n \to +\infty}\frac{\log|DF_{0}^{n}(x)|}{n}}
={\limsup_{n \to +\infty}\frac{\log|DF_{0}^{k_{n}m!+l_{n}}(x)|}{k_{n}m!+l_{n}}}\\
&\leq& {\limsup_{n \to +\infty}\frac{\log(L\cdot|DF_{0}^{k_{n}m!}(x)|)}{k_{n}m!+l_{n}}}=\frac{1}{m!}{\limsup_{k \to +\infty}\frac{\log|DT^{k}(x)|}{k}}\\
&=&\frac{1}{m!}\lambda_{1}(x,T)\leq0,
\end{eqnarray*}
where $L:=\max\{|DF_{0}^{l}(u)|:\ 0\leq l\leq m!-1,\ u\in\overline{O(x,F_{0})}\}$. Thus, we have proved the assertion.

Define
$$S_{0}:=\{x\in X:\ \lambda_{\max}(x,F_{0})> 0\}.$$
Recall that $F_{0}$ is $C^1$. Then the function $x\mapsto \frac{\log\|DF_0^n(x)\|}{n}$ is Borel-measurable, for any $n\geq1$. Hence, $x\mapsto \lambda_{\max}(x,F_0)$ is Borel-measurable. Thus, $S_{0}$ is Borel.

Now, we define
$$Q_{0}^{\prime}:=\{x\in X:\omega(x,T)\text{ is a linearly stable fixed point}\}.$$
The assertion and \eqref{E:map-lambda 1-geq-lambda max} entails that $Q_{0}^{\prime}\subset S_{0}^c$, where $S_{0}^c$ denotes its complement $X\backslash S_{0}$. It follows from the proof of Proposition \ref{T:im-pre} that $Q_{0}^{\prime}$ is prevalent in $X$. Then, we obtain $S_{0}^c$ is also prevalent in $X$, which implies that $S_{0}$ is shy.

In particular, if $X=\mathbb{R}^{d}$, $S_{0}$ is of Lebesgue measure zero  in $\mathbb{R}^{d}$. It follows that $F_{0}$ cannot have observable chaos. Thus, we have proved Theorem A.\hfill $\square$

\vskip 3mm

\noindent
{\it Proof of Theorem B.}
Let the open bounded subsets $D_1,D_2$ (with $D_1\supset\overline{D}_2\supset D_{2}\supset A$) and the integer $q>0$ be in Lemma \ref{P:im-da-perturb} and Remark \ref{R:bdd-per-Fep}. Define the closed bounded set $M:=\overline{D}_2$. It is noticed that $F_{\ep}^{q}(M)\subset M$, since $F_{\ep}^{q}(D_2)\subset D_2$. Define
$$S_{\ep,D_1}:=\{x\in D_1:\  \lambda_{\max}(x,F_{\ep}^{q})> 0\}.$$
By virtue of \eqref{E:Fep-m-pre}, one can follow the exactly same proof of Theorem A to obtain that, $S_{\ep,D_1}$ is shy. Now, we define
$$S_{\ep}:=\{x\in M:\ \lambda_{\max}(x,F_{\epsilon})> 0\}.$$
Clearly,
$$S_{\ep}=\{x\in M:\ \lambda_{\max}(x,F_{\ep}^q)> 0\}.$$
Then, $S_{\ep}\subset  S_{\ep,D_1}$. Moreover, the same reason of $S_{0}$ is Borel in the proof of Theorem A also entails that, $S_{\ep}$ is Borel. Therefore, $S_{\ep}$ is also shy.

In particular, if $X=\mathbb{R}^{d}$, then $M$ is a compact domain in $\mathbb{R}^{d}$. The shyness of $S_{\ep}$ implies that $S_{\ep}$ is of Lebesgue measure zero in $M$. It follows that $F_{\ep}^{q}|_M$ cannot have observable chaos. The proof is completed.

\vskip 3mm

\noindent{\it Proof of Theorem C.}
Define the sets
$$C_{s}:=\{x\in X:\ \omega(x,\Phi)\ \textnormal{is a linearly stable equilibrium}\},$$
and
$$S_{\Phi}:=\{x\in X:\ \lambda_{\max}(x,\Phi)> 0\}.$$

We are going to prove $C_{s}\subset S_{\Phi}^c$. Suppose that $x\in C_{s}$ and $\om(x,\Phi)=\{e\}$. Let $\tilde{t}>0$. Then $r_{\sigma}(D\Phi_{\tilde{t}}(e))\leq1$. Thus, the assertion in the proof of Theorem A (with $F_{0}$ and $T$ replaced by $\Phi_{\tilde{t}}$ there) entails that, $\lambda_{1}(x,\Phi_{\tilde{t}})\leq0$. Thus, $\lambda_{1}(x,\Phi)\leq0$ since the chain rule and the compactness of $\overline{O(x,\Phi)}$. Hence, it follows from \eqref{E:flow-lambda 1-geq-lambda max} that $\lambda_{\max}(x,\Phi)\leq0$. That is to say, $C_{s}\subset S_{\Phi}^c$.

Recall that $\Phi$ is a $C^1$-semiflow. Then $x\mapsto \lambda_{\max}(x,\Phi)$ is Borel-measurable. Thus, $S_{\Phi}$ is Borel. Enciso, Hirsch and Smith \cite[Theorem 1]{EHS08} entails that, $C_{s}$ is prevalent in $X$. Therefore, $S_{\Phi}^c$ is also prevalent in $X$,  which implies that $S_{\Phi}$ is shy.

In particular, if $X=\mathbb{R}^{d}$, $S_{\Phi}$  is of Lebesgue measure zero  in $\mathbb{R}^{d}$.
It follows that $\Phi$ cannot have observable chaos. Thus, we have proved Theorem C.\hfill $\square$

\begin{rmk}\label{R:im-pre-tran-chaos}
Theorems A-C confirm that there is no observable chaos (in the sense of Young \cite[Section 1]{Y13}) in monotone dynamical systems, as well as their $C^1$-perturbed systems. It has been known that chaotic dynamics, including horseshoes, can be found in the codimension-1 invariant manifolds in monotone systems (see e.g., \cite{DP98,S76,S88,S95,S97,WJ01,WX10} and references therein). As a matter of fact, for the horseshoe itself occupies a zero measure set, its presence does not conflict with the fact that,
orbits starting from a full measure set tend eventually to a stable fixed point for certain iteration of a strongly monotone map, as well as for its $C^1$-perturbed maps (see Proposition \ref{T:im-pre} and Proposition \ref{T:im-pre-per}).
Hence, we point out that those chaotic behavior in monotone dynamical systems and their $C^1$-perturbed systems are indeed transient chaos (in the sense of Young \cite[Section 1]{Y13}).
\end{rmk}

\end{document}